\documentclass[11pt, oneside]{article}   	
\usepackage{geometry}                		
\usepackage{graphicx}				
\usepackage{amssymb}
\usepackage{amsmath}
\usepackage{amsthm}
\usepackage{verbatim}
\usepackage{harpoon}
\usepackage{accents}
\newcommand{\vect}[1]{\accentset{\rightharpoonup}{#1}}
\usepackage{tikz}  
\usepackage{mathrsfs}
\usepackage{verbatim}

\usetikzlibrary{positioning,chains,fit,shapes,calc}  

\newtheorem{theorem}{Theorem}
\newtheorem{lemma}{Lemma}
\newtheorem{definition}{Definition}

\newtheorem{example}{Example}

\title{Verification of the Jacobian Conjecture for $d$-linear maps in two variables}
\author{Mario DeFranco}

\begin{document} 

\maketitle

\abstract{For any integer $d \geq 1$, we verify the Jacobian Conjecture for a $d$-linear map in two variables. We prove that almost all the coefficients of the formal inverse are in the ideal specified by the Jacobian condition. We find expressions for certain elements in terms of the generators of this ideal. To obtain these expressions, we generalize a bijective proof of the Cayley-Hamilton theorem in two ways.} 

\section{Introduction}\label{intro}
In this paper we verify the Jacobian conjecture for $d$-linear maps in two variables for any integer $d \geq 1$. 
This paper generalizes \cite{DeFranco} and uses much of the same notation and terminology.
 
The Jacobian Conjecture is a problem in algebraic geometry first stated by Ott-Heinrich Keller in 1939 (see \cite{Keller} and van den Essen \cite{van den Essen}). It has several formulations, and in this paper we consider the following.

For an integer $n \geq 1$, let $x = (x_1, \dots, x_n)\in \mathbb{C}^n$, define a polynomial map to be a map of the form  
\begin{align*}
&f \colon \mathbb{C}^n \rightarrow \mathbb{C}^n\\ 
&(x_1, \ldots, x_n) \mapsto (f_1(x_1, \ldots, x_n), \ldots, f_n(x_1, \ldots, x_n)  )
\end{align*}
where each component function $f_i(x_1, \ldots, x_n)$ is a polynomial.

\textbf{Jacobian Conjecture}: Let $n \geq 2$. Suppose $f(x)$ is a polynomial map such that 
\begin{equation}\label{Jac condition}
\mathrm{Jac}(f)(x) := \det(D(f)(x)) = 1 
\end{equation}
for all $x \in \mathbb{C}^n$, where $\mathrm{Jac}(f)$ is the Jacobian of $f(x)$ and $D(f)(x)$ is the differential of $f(x)$. Then $f(x)$ has an inverse function that is also a polynomial map.  

\begin{definition}
For integers $d,n\geq1$, call a map $f$ 
\[
f \colon \mathbb{C}^n \rightarrow \mathbb{C}^n
\]
a $d$-linear map if the $i$-th component function is of the form 
\[
f_i(x_1, \ldots, x_n) = x_i- (t\sum_{j=1}^n a_{i,j}x_j)^d
\]
for $t,a_{i,j}\in \mathbb{C}$, $t \neq 0$. 
\end{definition}

Instead of elements of $\mathbb{C}$, we treat $a_{i,j}$ and $t$ as indeterminates. Define $R$ to be the polynomial ring generated by the indeterminates $a_{i,j}, 1 \leq i,j\leq n$ with coefficients in $\mathbb{Q}$. The Jacobian $\mathrm{Jac}(f)(x)$ is then an element of $R[t,x]$, and the formal inverse of $f(x)$ is then a formal power series $g(x) \in R[[t,x]]$ (Lemma 1, \cite{DeFranco}). Define $J_{d,n}$ to be the ideal of $R$ generated by the coefficients of the monomials in $x,t$ of $\mathrm{Jac}(f)(x)$, omitting the constant term 1. The indeterminate $t$ is needed in the case $d=1$ to ensure that the coefficients of $g(x)$ are elements of $R$; for $d\geq2$ it is not needed.

Now we describe the layout of this paper. In Section \ref{verification}, Theorem \ref{t d lin n=2},  we prove that for a $d$-linear map $f(x)$ with $n=2$, if $\mathrm{Jac}(f)(x)=1 \in R[x]$, then all but finitely many of the coefficients of its formal inverse $g(x)$ are in $J_{d,2}$. In Section \ref{gen CH}, Theorem \ref{t generalization CH},  we prove two generalizations of the Cayley-Hamilton theorem used to prove Theorem \ref{t d lin n=2} which recover the Cayley-Hamilton theorem when $d=1$. Theorem \ref{gen CH} generalizes Theorem 2 of \cite{DeFranco} which is a combinatorial proof similar to that of Straubing \cite{Straubing}. 

References to the Cayley-Hamilton in connection to the Jacobian Conjecture appear in Abdesselam \cite{Abdesselam} and Umirbaev \cite{Umirbaev}. Abdesselam remarks in \cite{Abdesselam} that the Cayley-Hamilton theorem implies the Jacobian conjecture when $``d=1"$. 

\section{Generalizations of Cayley-Hamilton Theorem} \label{gen CH}

\begin{definition}
Let $C(m,n)$ denote the set of compositions $\alpha$ 
\[
\alpha = (\alpha(1), \ldots, \alpha(n))
\] 
of $m$ into $n$ non-negative integer parts. Denote 
\[
x^\alpha = \prod_{i=1}^n x_i^{\alpha(i)}.
\]
Suppose $m_1+m_2=m$. Then for $\alpha_1 \in C(m_1,n)$ and $\alpha_2 \in C(m_2,n)$, we say 
\[
\alpha_1+\alpha_2 = \alpha \text{ and } \alpha - \alpha_2 = \alpha_1
\] 
if for $1 \leq i \leq n$
\[
\alpha_1(i) +\alpha_2(i) = \alpha(i).
\]

Fix integers $d,n\geq1.$ For an integer $k\geq 0$, define a $k$-level labeling $\nu$ to be a $k$-tuple
\[
\nu = (\nu(1), \ldots, \nu(k))
\]
where $\nu(i)$ is a $(d-1)$-tuple
\[
\nu(i) = (\nu(i,1), \ldots, \nu(i,d-1))
\]
such that $\nu(i,j) \in [1,n]$. For $\alpha \in C(k(d-1),n)$, let $I(\alpha,k)$ denote the set of $k$-level labelings $\nu$ such that 
\[
\alpha(r) = \#\{(i,j)\colon \nu(i,j) = r \}.
\] 

Suppose $S \subset [1,n]$ of order $k$, $\sigma$ a permutation of $S$, $\alpha \in  C(k(d-1),n)$, and $ \nu \in I(\alpha,k)$. Define $w(\sigma, \nu)$ 
\[
w(\sigma, \nu) =(-1)^{\# \mathrm{cycles}(\sigma)} \prod_{i=1}^{k}  (a_{S(i),\sigma(S(i))} \prod_{j=1}^{d-1} a_{S(i),\nu(i,j)}).
\]

\end{definition}
The differential $D(f)$ of a $d$-linear map $f(x)$ is the $n \times n$ matrix 
\begin{equation} \label{Df}
(D(f))_{i,j} = \delta_{i,j} - t da_{i,j} (t\sum_{r=1}^n a_{i,r}x_r)^{d-1}.
\end{equation}
Define the elements $J_\alpha \in R$ by 
\begin{equation} \label{Jac matrix}
\det(D(f)) = \sum_{k=0}^n d^kt^{dk} \sum_{\alpha \in C(k(d-1),n)} J_\alpha x^\alpha.
\end{equation}

\begin{lemma}\label{l J alpha form}
For $\alpha \in C(k(d-1),n)$,
\[
J_{\alpha} = \sum_{S \subset [1,n], |S|=k} \sum_{\sigma} \sum_{\nu \in I(\alpha,k)} w(\sigma, \nu).
\]
where $\sigma$ is a permutation of $S$. 
\end{lemma}
\begin{proof}
Expand the determinant \eqref{Jac matrix} as sum over the permutation group on $n$ elements. Each term is indexed by a choice of $n-k$ rows from which 1 is chosen from the diagonal entry. The remaining $k$ rows determine the set $S$. The permutation $\sigma$ is the restriction of the permutation from the determinant expansion to $S$. The sign of this term is 
\begin{align*}
(-1)^k \mathrm{sgn}(\sigma) &= (-1)^k \prod_{c  \in \mathrm{cycles}(\sigma)} (-1)^{1+ \mathrm{length}(c)}\\
&=(-1)^{\# \mathrm{cycles}(\sigma)}.
\end{align*}
The factor 
\[
\prod_{i=1}^{k} a_{S(i),\sigma(S(i))}
\]
in the definition of $w(\sigma,\nu)$ comes from the of $a_{i,j}$ outside the sum from the entry \eqref{Df}. From row $S(i)$ we expand
\[
(\sum_{r=1}^n a_{S(i),r}x_r)^{d-1} = \prod_{h=1}^{d-1}(\sum_{r=1}^n a_{S(i),r}x_r).
\]
into monomials in $x$, choosing $x_{\nu(i,h)}$ from the $h$-th factor. This completes the proof. 
\end{proof}

\begin{definition}

For a $k$-tuple $\lambda= (\lambda(0),\ldots, \lambda(k))$, define the substrings $C_1(\lambda)$ and $C_2(\lambda)$
\begin{align*}
C_1(\lambda) &= (\lambda(l_1), \ldots, \lambda(l_2-1))\\
C_2(\lambda) &= (\lambda(l_1+1), \ldots, \lambda(l_2))
\end{align*}
such that $\lambda(l_1) = \lambda(l_2), \lambda(i) \neq \lambda(j)$ for $l_1<i< j$, if such a string exists. We say that $(l_1,l_2)$ are the last-rep indices of $\lambda$. 
\end{definition}
In \cite{DeFranco} we called the above the ``first-rep indices".

Below we use the definition of the element $z(\mathrm{fern}_{d,n-k},( u_{0}, u_{n}; \nu))$ from \cite{DeFranco}. Here $(u_{0}, u_{n}; \nu)$ is a root-leaf labeling of $\mathrm{fern}_{d,n-k}$ where the root $v_0$ is labeled $u_0$, the leaf vertex $v_{n-k}$ on the leftmost path is labeled $u_n$, and the $j$-th child vertex of $v_i$ to the right of $v_{i+1}$ is labeled by $\nu(i+1,j)$.
 
\begin{theorem}\label{t generalization CH}
For integers $d,n \geq 1$, let $u_{0}, u_{n} \in [1,n]$. 

For $n\geq 1$ and $\alpha \in C(n(d-1),n)$
\begin{equation}\label{generalization CH 1}
\sum_{k=0}^n \sum_{\alpha_1 \in C(k(d-1),n)} \sum_{\nu \in I(\alpha-\alpha_1,k)}z(\mathrm{fern}_{d,k}, (u_{0}, u_{n}; \nu)) J_{\alpha_1}=0.
\end{equation}

For $n\geq 2$, fix a $(d-1)$-tuple $\beta$ of integers in $[1,n]$.
For $\alpha \in C(n(d-1),n)$ and $u_{0} \neq u_{n}$,
\begin{equation}\label{generalization CH 2}
\sum_{k=0}^{n-1} \sum_{\alpha_1 \in C(k(d-1),n)} \sum_{\nu \in I(\alpha-\alpha_1,n-k), \nu(1)=\beta}z(\mathrm{fern}_{d,n-k},( u_{0}, u_{n}; \nu)) J_{\alpha_1}=0.
\end{equation}

\end{theorem}
\begin{proof} 
We expand each term in the sums \eqref{generalization CH 1} or \eqref{generalization CH 2} into a sum monomials in $R$ and construct a sign-reversing involution $I$ on those monomials. 

Consider the sum  \eqref{generalization CH 1}. In the expansion of $\displaystyle z(\mathrm{fern}_{d,n-k},( u_0,u_n; \nu)) J_{\alpha_1}$, we index each monomial by $(\lambda,\nu, S, \sigma, \rho)$, where $\lambda$ is a $(n-k+1)$-tuple $\lambda$ of integers in $[1,n]$ 
\[
\lambda = (\lambda(0), \ldots, \lambda(n-k))
\]
such that $\lambda(h)$ is the label of the $h$-th vertex on the leftmost path of $\mathrm{fern}_{d,n-k}$; $\nu$ is an $(n-k)$-level labeling of $\mathrm{fern}_{d,n-k}$; $S$ is a subset of $[1,n]$ of order $k$; $\sigma$ is a permutation of $S$; and $\rho$ is a $k$-level labeling. If $k=0$ then $S$ and $\sigma$ are empty. Denote the set of all such $(\lambda,\nu, S, \sigma, \rho)$ by $K(u_0,u_n)$. 

Each of the substrings $C_1(\lambda)$ and $C_2(\lambda)$, being a sequence of distinct integers, corresponds to a cycle of a permutation. 
One direction of the involution is constructed by transferring $C_1(\lambda)$ or $C_2(\lambda)$ from $\lambda$ to the permutation $\sigma$. When $d\geq 2$, we describe two ways to do this, which we call $\tau_1$ and $\tau_2$. 

Let $(l_1,l_2)$ be the last-rep indices for $\lambda$.

Suppose $l_2 \leq n-k$. Let $\lambda'$ denote 
\[
\lambda' = (\lambda(0), \ldots, \lambda(l_1-1), \lambda(l_2),\ldots \lambda(n-k))
\]
and $\nu'$ denote
\[
\nu'=(\nu(1), \ldots, \nu(l_1), \nu(l_2+1), \ldots, \nu(n-k)).
\]
Adjoin this cycle $C_1(\lambda)$ to $\sigma$ to obtain a permutation $\sigma'$ on the resulting set $S'$. 

Now for each $i$, $S'(i)$ is either in $S$ or $C_1(\lambda)$. We define the index $e_1(i)$ for $1 \leq i \leq |S'|=k+l_2-l_1$ by

 \begin{align*} 
&S'(i) = S(e(i)) \text{ if } S'(i) \in S \\  
&S'(i)= \lambda(e_1(i))  \text{ if } S'(i) \in C_1(\lambda) \text{ such that } l_1 \leq e(i) <l_2.
\end{align*}
Define $\rho'$ by
 
 \[
\rho'(i) = \begin{cases} 
&\rho(e(i)) \text{ if } S'(i) \in S \\ 
 & \nu(e_1(i)+1) \text{ if } S'(i) \in C_1(\lambda).
\end{cases}
\]
Then define
\[
\tau_1((\lambda,\nu, S, \sigma, \rho)) = (\lambda',\nu', S', \sigma', \rho').
\]

Now suppose $l_2 < n-k$. Let $\lambda'$ denote 
\[
\lambda' = (\lambda(0), \ldots, \lambda(l_1), \lambda(l_2+1),\ldots \lambda(n-k))
\]
and $\nu''$ denote
\[
\nu''=(\nu(1), \ldots, \nu(l_1+1), \nu(l_2+2), \ldots, \nu(n-k)).
\]
Adjoin the cycle $C_2(\lambda)$ to $\sigma$ to obtain a permutation $\sigma''$ on the resulting set $S''$. 
Now for each $i$, $S''(i)$ is either in $S$ or $C(\lambda)$. Define the index $e_2(i)$ by 
 \begin{align*} 
&S'(i) = S(e_2(i)) \text{ if } S'(i) \in S \\ 
&S'(i)= \lambda(e_2(i))  \text{ if } S'(i) \in C_2(\lambda)  \text{ such that } l_1 < e(i) \leq l_2.
\end{align*}
Define $\rho''$ by
 
 \[
\rho''(i) = \begin{cases} 
&\rho(e_2(i)) \text{ if } S'(i) \in S \\ 
 & \nu(e_2(i)+1) \text{ if } S'(i) \in C_2(\lambda).
\end{cases}
\]
Then define
\[
\tau_2((\lambda,\nu, S, \sigma, \rho)) = (\lambda'',\nu'', S'', \sigma'', \rho'').
\]

If $l_2=n-k$, then define $\tau_2$ by 
\[
\tau_2((\lambda,\nu, S, \sigma, \rho))=\tau_1((\lambda,\nu, S, \sigma, \rho)).
\]

Let $h$ be the greatest integer such that $\lambda(h)\in S$. Then 
we claim at least one of $h$ or $(l_1,l_2)$ exists. For if $(l_1,l_2)$ does not exist, then the set of $n-k+1$ integers of $\lambda$ are distinct, and if $h$ also does not exist, then this set and $S$ (which consists of $k$ distinct integers by construction) are disjoint. This contradicts the fact that the integers are in $[1,n]$. This proves the claim. 

Now for a given $(\lambda,\nu, S, \sigma, \rho)$, suppose $(l_1, l_2)$ exists and either $h$ does not exist or $h<l_1$. Then we claim that $(\lambda,\nu, S, \sigma, \rho)$ is uniquely determined by either one of the images $\tau_1((\lambda,\nu, S, \sigma, \rho))$ or 
$\tau_2((\lambda,\nu, S, \sigma, \rho))$. Denote such an image by $(\lambda',\nu', S', \sigma', \rho')$. 
Let $b$ be the integer of $S'$ such that $b=\lambda'(h')$ with $h'$ maximal. Then we must have $b=\lambda(l_1)$ and $h'>h$ (if $h$ exists). The cycle that was transferred to make $\sigma'$ is thus the cycle that contains $b$. There is only one way to insert this cycle into $\lambda'$ (yielding $\lambda$) and two ways (depending on the choice of $\tau_1$ or $\tau_2$) to transfer level labelings from $\rho'$ to $\nu'$ (yielding $\nu$). Note that if $h'=n-k$ then these two ways coincide. 

Therefore we define the domain and images of $\tau_1$ and $\tau_2$ as
\begin{align*}
&\mathrm{Domain}(\tau) &= \{(\lambda,\nu, S, \sigma, \rho)\colon (l_1, l_2) \text{ exists and either $h$ does not exist or $h<l_1$} \}\\ 
&\mathrm{Image}(\tau) &= \{(\lambda,\nu, S, \sigma, \rho)\colon h \text{ exists and either $(l_1,l_2)$ does not exist or $h>l_1$} \}
\end{align*}
Note that in the image we cannot have $l_1=h$, for then $\lambda(l_2) = \lambda(h)$, contradicting the maximality of $h$. The domain and image are disjoint, and their union is the set $K$ of all $(\lambda,\nu, S, \sigma, \rho)$. From the preceding paragraph, for $i=1,2$
\[
\tau_i\colon \mathrm{Domain}(\tau) \rightarrow  \mathrm{Image}(\tau)
\] 
is a bijection. Define for $i=1,2$
\begin{align*}
I_i(( \lambda, \nu, S, \sigma, \rho)) &= \tau_i(( \lambda, \nu, S, \sigma, \rho)) \text{ if } ( \lambda, \nu, S, \sigma, \rho) \in \mathrm{Domain}(\tau)\\ 
I_i(( \lambda, \nu, S, \sigma, \rho)) &= \tau_i^{-1}(( \lambda, \nu, S, \sigma, \rho)) \text{ if } ( \lambda, \nu, S, \sigma, \rho) \in \mathrm{Image}(\tau).
\end{align*}

Applying either $I_1$ or $I_2$ to the set $K(u_0,u_n)$ yields equation \eqref{generalization CH 1}. Let $K(u_0,u_n;\beta) \subset K(u_0,u_n)$ denote the set of $( \lambda, \nu, S, \sigma, \rho)$ where $\nu(1) = \beta$ and $|S|\neq n$. Now if $u_0\neq u_n$, then 
\[
\tau_2 \colon \mathrm{Domain}(\tau)\cap K(u_0,u_n;\beta) \rightarrow  \mathrm{Image}(\tau)\cap K(u_0,u_n;\beta)
\]
is also a bijection. For any $(l_1,l_2)$ must have $l_2\geq1$, and thus $\nu(1)$ is never transferred to $\sigma$. Applying $I_2$ to the set $K(u_0,u_n;\beta)$ yields equation \eqref{generalization CH 2}. This completes the proof. 


\end{proof} 

\begin{example}\label{e n=2}
\emph 
{Consider the element $z(\mathrm{fern}_{d,2}, (1,2; \nu))$ for some 1-level labeling $\beta$ and a composition $\alpha$ with 
\begin{align*}
\alpha(1) - \#\{i \colon \beta(1,i) = 1\} &= m \\ 
\alpha(2) -\#\{i \colon \beta(1,i) = 2\}&= d-1-m.
\end{align*}
for some $m$. By Theorem \ref{t generalization CH},  equation \eqref{generalization CH 2}, the $k=0$ term is 
\[
{d-1 \choose m} \left( (a_{1,1}a_{1,2})( \prod_{j=1}^{d-1}a_{1,\beta(1,j)}) a_{1,1}^{m}a_{1,2}^{d-1-m} + (a_{1,2}a_{2,2})( \prod_{j=1}^{d-1}a_{1,\beta(1,j)}) a_{2,1}^{m}a_{2,2}^{d-1-m} \right).
\]
This expression is equal to 
\[
{d-1 \choose m} z(\mathrm{fern}_{d,2}, (1,2; \nu))
\]
where $\nu$ is a 2-level labeling with $\nu(1) = \beta$ and $m = \#\{i \colon \nu(2,i) = 1\}$. }

\emph{In the $k=1$ term, in order for $\nu(1)=\beta$, we must have $\alpha_1(1) = m$. With this $\alpha_1$ the $k=1$ term is then
\[
(a_{1,2})( \prod_{j=1}^{d-1}a_{1,\beta(1,j)}) J_{\alpha_1}.
\]
 }
\end{example}

\section{Proof of Verification}\label{verification}

\begin{lemma}\label{l one 2}
Let $n=2, d \geq 1$.  Let $\nu$ be a $2$-level labeling. Suppose either $u_2=2$ or $\nu(2,j)=2$ for some $1\leq j\leq d-1$. Then  
\[
z(\mathrm{fern}_{d,2}, (1,u_2; \nu)) \in J_{d,2}.
\] 
\end{lemma}
\begin{proof}
Theorem \ref{t generalization CH}, equation \eqref{generalization CH 2} (see Example \ref{e n=2}) implies that 
\begin{equation}\label{z 12}
z(\mathrm{fern}_{d,2}, (1,2; \mu)) \in J_{d,2}
\end{equation}
for any 2-level labeling $\mu$. Thus suppose $u_2=1$. Let $\nu'$ be the 2-level labeling obtained by replacing $\nu(2,j)$ with 1. Then 
\[
z(\mathrm{fern}_{d,2}, (1,1; \nu)) = z(\mathrm{fern}_{d,2}, (1,2; \nu')) 
\]
and the right side is in $J_d,2$  as an instance of equation \eqref{z 12}. This completes the proof. 
\end{proof}

\begin{lemma}\label{z elements}
Let $n=2$ and $d \geq 1$. For any $u_0,u_2 \in \{ 1,2\}$ and any $2$-level labeling $\nu$, 
\begin{equation}\label{d lin n=2}
z(\mathrm{fern}_{d,2}, (u_0,u_2; \nu))\in J_{d,2}.
\end{equation}
\end{lemma}
\begin{proof}
 Without loss of generality we assume that $u_0=1$. 
 
 If $u_2=2$ or any $\nu(2,j)=2$, by Lemma \ref{l one 2} statement \eqref{d lin n=2} holds. 
 
 If $u_2=1$ and $\nu(1,j) = \nu(2,j)=1$ for $1 \leq j \leq d-1$, then the $k=0$ term of equation \eqref{generalization CH 1} is 
 \begin{equation}\label{all 1}
 z(\mathrm{fern}_{d,2}, (1,1; \nu))
 \end{equation}
 Each of the other terms with $k\geq1$ is in $J_{d,2}$. Thus statement \eqref{d lin n=2} holds. 
 
 Suppose that $u_2=1$ and $\nu(2,j)=1$ for all $j$, and there is at least one $j$ such that  $\nu(1,j)=2$. Then the $k=0$ term of $\eqref{generalization CH 1}$ is equal to 
 \[
 z(\mathrm{fern}_{d,2}, (u_0,u_2; \nu)) + \sum_{\nu'} z(\mathrm{fern}_{d,2}, (u_0,u_2; \nu'))
 \] 
 where each $\nu'$ in the sum has $\nu'(2,i)=2$ for some $i$. By Lemma \ref{l one 2} each such term in the sum is in $J_{d,2}$. The terms for $k \geq 1$ are also in $J_{d,2}$. Thus statement \eqref{d lin n=2} holds. This completes the proof. 
\end{proof}

\begin{theorem}\label{t d lin n=2}
For a $d$-linear map $f(x)$ with $n=2$, suppose $\mathrm{Jac}(f(x)) = 1$ for all $x \in \mathbb{C}^2$. Then the formal inverse $g(x)$ of $f(x)$ is a polynomial map. 
\end{theorem}
\begin{proof} 
We use the notation of \cite{DeFranco}, Section 2. Consider the coefficient $c_{\alpha,N} \in R$ of the monomial $t^N x_1^{\alpha(1)}x_2^{\alpha(2)}$ in the formal power series for $g_i(x), i \in \{ 1,2\}$. By Lemma 1 of \cite{DeFranco}, $c_{\alpha}$ is a sum over labeled trees $(T,\lambda)$ with $N$ edges, where the root is labeled $i$, and there are $\alpha(1)$ leaves labeled 1 and $\alpha(2)$ leaves labeled 2. Since each vertex of such a tree has degree $d$ or degree $0$, we have for each tree $T$
\[
\alpha(1)+ \alpha(2) = 1+(d-1)\frac{N}{d}.
\]
There are finitely many trees contributing to $c_{\alpha,N}$. Furthermore there are only finitely many pairs ($N$, $\alpha$) such that all trees $T$ contributing to $c_{\alpha,N}$ have height at most 1. Thus we may assume each $T$ has height at least two. Now partition the set of the labeled trees $(T,\lambda)$ contributing to $c_{\alpha,N}$ according to the their underlying trees $T$, and then further partition into subsets $K(T,\lambda)$. By Lemma 2 of \cite{DeFranco}, the sum over $K(T,\lambda)$ has a factor of the form 
\[
z(\mathrm{fern}_{d,2}, (i,u_2; \nu))
\]
for some $u_2 \in \{1,2 \}$ and some 2-level labeling $\nu$. By Lemma \ref{z elements}, this element is in $J_{d,2}$. This completes the proof.
\end{proof}

We describe another relation that may be used in the proof of Lemma \ref{z elements}. Let $\alpha_1, \alpha_2 \in C(d-1, 2)$ such that $\alpha_1(1)\geq 1$. Let $u \in \{ 1,2\}$ and define $\alpha_1'$ and $\alpha_1';$ be the compositions
\begin{align*}
\alpha_1' &= (\alpha_1(1)-\delta_{u,2}, \alpha_1(2) + \delta_{u,2})\\
\alpha_1'' &= (\alpha_1(1)-1, \alpha_1(2) +1).
\end{align*}
Then it is straightforward to check that
\begin{equation} \label{2 1s}
z(\mathrm{fern}_{d,2}, (u,v; \nu)) = \frac{a_{2,u}a_{2,1}^{\alpha_2(1)}a_{2,2}^{\alpha_2(2)} }{{d-1, \choose \alpha_1(1)-1} }J_{\alpha_1''} - \frac{a_{2,2}a_{2,1}^{\alpha_2(1)}a_{2,2}^{\alpha_2(2)} }{{d-1, \choose \alpha_1(2)+ \delta_{u,2}} }J_{\alpha_1'} 
+\frac{a_{1,u}a_{1,1}^{\alpha_1(1)}a_{1,2}^{\alpha_1(2)} }{{d-1, \choose \alpha_2(1)} }J_{\alpha_2} 
\end{equation}
where $\nu$ is a 2-level labeling such that $\nu(1)$ contains exactly $\alpha_1(1)$ 1's and $\nu(2)$ contains exactly $\alpha_2(1)$ 1's. 

\section{Further Work}
\begin{itemize}

\item Since the elements $z(\mathrm{fern}_{d,n}, (u_0,u_n; \nu))$ are homogeneous and $J_{d,n}$ is generated by homogenous elements, ideal membership may be reduced to vector space membership. See if there are formulas for the row-reduced matrix in terms of $d$ and $n$. 

\item See if equation \eqref{2 1s} generalizes to higher $n$ or if it can be expressed using the bijections $I_1,I_2$ or others. 

\item See if the operations $\tau_1$ and $\tau_2$ can be mixed to form intermediate involutions between $I_1$ and $I_2$, or if other operations can be found. 

\item Characterize all relations among the $J_\alpha$ with coefficients of the form $z(\mathrm{fern}_{d,k}, (i,j;\nu)$.

\item The operations $\tau_1$ and $\tau_2$ arise from the fact that an interval has two boundary components. See if the bijections can be generalized by generalizing the interval to higher dimensions. 

\item See if these generalizations of the Cayley-Hamilton Theorem apply to the algebras of Umirbaev \cite{Umirbaev}. 

\end{itemize}

\end{document}